\documentclass{article}
\usepackage{amsrefs}
\usepackage{amssymb}
\usepackage{amsmath}
\usepackage{amsthm}
\usepackage{graphicx}

\newtheorem{theorem}{Theorem}
\newtheorem{proposition}{Proposition}
\newtheorem{lemma}{Lemma} 
\newtheorem{definition}{Definition}
\newtheorem{corollary}{Corollary}
\newtheorem*{theoremA}{Theorem A}
\newtheorem*{theoremB}{Theorem B}
\newtheorem{example}{Example}
\DeclareMathOperator{\WS}{WS}
\newtheorem*{subject}{2000 Mathematics Subject Classification}
\newtheorem*{keywords}{Keywords}

\author{Marc Coppens\footnote{KU  Leuven, Technologiecampus Geel, Departement Elektrotechniek (ESAT),
Kleinhoefstraat 4, B-2440 Geel, Belgium; email: marc.coppens@kuleuven.be.}}
\title{The uniqueness of Weierstrass points with semigroup $<a;b>$ and related semigroups. }
\date{}

\begin{document}
\maketitle \noindent

\begin{abstract}
Assume $a$ and $b=na+r$ with $n \geq 1$ and $0<r<a$ are relatively prime integers.
In case $C$ is a smooth curve and $P$ is a point on $C$ with Weierstrass semigroup equal to $<a;b>$ then $C$ is called a $C_{a;b}$-curve.
In case $r \neq a-1$ and $b \neq a+1$ we prove $C$ has no other point $Q \neq P$ having Weierstrass semigroup equal to $<a;b>$.
We say the Weierstrass semigroup $<a;b>$ occurs at most once.
The curve $C_{a;b}$  has genus $(a-1)(b-1)/2$ and the result is generalized to genus $g<(a-1)(b-1)/2$.
We obtain a lower bound on $g$ (sharp in many cases) such that all Weierstrass semigroups of genus $g$ containing $<a;b>$ occur at most once.
\end{abstract}

\begin{subject}
14H55
\end{subject}

\begin{keywords}
Weierstrass points, gonality, Weierstrass semigroup, $C_{a,b}$-curves
\end{keywords}

\section{Introduction}\label{section1}

We write $\mathbb{N}$ to denote the semigroup of non-negative integers (in particular including 0).
A subsemigroup $H$ of $\mathbb{N}$ is called a Weierstrass semigroup of genus $g$ if the complement $\mathbb{N} \setminus H$ is a finite set of exactly $g$ integers.
Let $C$ be a smooth curve of genus $g$ and let $\mathcal{O}_C$ be the sheaf or regular functions on $C$.
Let $P$ be a point on $C$ and consider $\{ \deg (f) : f \in \mathcal{O}_C(C\setminus \{ P\}) \}$.
This is a Weierstrass semigroup of genus $g$ called the Weierstrass semigroup of $P$ and denoted by $\WS (P)$.
In case $ f \in \mathcal{O}_C(C\setminus \{ P\})$ is not a constant then it defines a morphism $f : C \rightarrow \mathbb{P}^1$ with $f^{-1}(\infty )=\{ P \}$ and introducing multiplicities for points on fibers of the morphism one obtains a base point free linear system $g^1_{\deg (f)}$ on $C$ containing the divisor $\deg (f)P$.
Therefore the Weierstrass semigroup of $P$ can also be described as follows
\[
\WS (P) = \{ a \in \mathbb{N} : \vert aP \vert \text { is a base point free linear system } \} \cup \{ 0 \} \text { .}
\]
The elements of $\mathbb{N} \setminus \WS (P)$ are called the gaps of $P$ (and the elements of $\WS (P)$ are called the non-gaps of $P$).
For all but finitely many points of $C$ the set of gaps of $P$ is equal to $\{ 1; 2; \cdots ; g \}$.
A point $P$ is called a Weierstrass point of $C$ in case the set of gaps of $P$ is different from $\{ 1; 2; \cdots ; g \}$.
(For a more detailled introduction see e.g. \cite {ref5} Section III-5.)

For a general curve $C$  the set of gaps of each Weierstrass point is equal to $\{ 1; 2; \cdots ;g-1; g+1 \}$.
The most special curves are the hyperelliptic curves, i.e. curves having a morphism $f : C \rightarrow \mathbb {P}^1$ of degree 2.
In case $g \geq 2$ such morphism is unique (if it exists) and the Weierstrass points are exactly the $2g+2$ ramification points of $f$. 
In this case the set of gaps of each Weierstrass point is equal to $\{1; 3; 5; \cdots ; 2g-1 \}$.
Hence the Weierstrass semigroup is the subsemigroup of $\mathbb{N}$ generated by 2 and $2g+1$ (denoted by $<2;2g+1>$).
It is the only Weierstrass semigroup of genus $g$ having first non-gap equal to 2.

From this point of view the next case is to consider Weierstrass points $P$ with first non-gap equal to three.
In this case the curve $C$ needs to have a base point free linear system $g^1_3$ containing $ 3P $, i.e. there exists a covering $f : C \rightarrow \mathbb{P}^1$ of degree 3 having $P$ as a total ramification point.
In case $g \geq 5$ then the linear system $g^1_3$ is unique.
However in general a $g^1_3$ does not need to have a total ramification point and if it has a total ramification point then in general it is unique.
Therefore the situation is different from the situation of hyperelliptic curve and the linear system $g^1_3$ does not determine all Weierstrass points on the curve.
Moreover in case there is a total ramification point $P$ then $\WS (P)$ is not completely determined by $g$ and in general not even by $f$
Therefore in case $f$ has at least two total ramification points then their Weierstrass semigroups can be different.

In \cite{ref6} all possibilities of combinations of Weierstrass semigroups with first non-gap equal to 3 that can occur on some fixed curve of genus $g \geq 5$ are determined.
In particular in case $P$ has Weierstrass semigroup $<3;3n+1>$ (in this case the genus of $C$ is equal to $3n$) then there is no other point $Q$ on $C$ with $\WS (Q)=<3;3n+1>$ (and this situation occurs).
It is mentioned at the introduction of \cite {ref3} that this fact is proved in \cite {ref7}.
It seems to me that this is not explicitly mentioned in that paper.
The computations in \cite {ref7} to obtain Theorem 6 of that paper imply that in case $C$ has genus $3n$ and there is a covering $f: C \rightarrow \mathbb{P}^1$ of degree 3 having $g+2$ total ramification points then exactly one of them has Weierstrass semigroup equal to $<3;3n+1>$.
From \cite{ref6} (and also from \cite{ref7}) it follows that for all other Weierstrass semigroups $H$ with first non-gap equal to 3 there exist curves $C$ having at least two points with Weierstrass semigroup equal to $H$.

We make the following definition
\begin{definition} \normalfont \label{definitieIntro}
Let $H$ be a Weierstrass semigroup of genus $g$. We say that $H$ occurs at most once in case there exists no curve $C$ of genus $g$ having two different Weierstrass points $P$ and $Q$ with $\WS (P)=\WS (Q)=H$.
\end{definition}
There is no Weierstrass semigroup with first non-gap equal to 2 that occurs at most once.
The Weierstrass semigroups with first non-gap equal to 3 are exactly the semigroups $<3;3n+1>$ with $n \geq 2$ an integer.
Its genus is equal to $3n$.

In \cite {ref3} the author gives a lot of Weierstrass semigroups $H$ of some genus $g$ with first non-gap some prime number $a$ that occur at most once.
As an example this result holds for semigroups $<a;ka-2>$ for any integer $k\geq 2$
More general from the arguments in \cite {ref3} it follows that for a prime number $a$ and $b=ka-r$ with $k \geq 2$ and $2 \leq r \leq a-1$ and $r \neq a-1$ in case $k=2$ there are at most $r-1$ Weierstrass points having Weierstrass semigroup equal to $<a;b>$  on a curve $C$ of genus $g=(a-1)(b-1)/2$ (this is indeed the number of non-gaps in case the Weierstrass semigroup is equal to $<a;b>$).
In case $r \neq 2$ this upper bound is not sharp.
In particular in \cite {ref4} , Theorem 1, it is proved that in case $a \geq 5$ is any odd integer then $<a;a+2>$ occurs at most once.
This is smaller than the bound obtained in \cite{ref3} in case $a \geq 5$ is a prime number.

One of the main results of this paper is the following theorem.

\begin{theoremA} \normalfont \label    {theoremA}
Let $a; b$ be relatively prime integers (we denote it by $(a;b)=1$) such that $b \geq a+2$.
Assume $b=ka+r$ with $1 \leq r \leq a-2$.
The Weierstrass semigroup $<a;b>$ (having genus $(a-1)(b-1)/2$) occurs at most once.
\end{theoremA}

In case $b=a+1$ or $r=d-1$ then there exist smooth curves of genus $(a-1)(b-1)/2$ having more than one Weierstrass point with Weierstrass semigroup equal to $<a;b>$.
The proofs in \cite{ref3} consist of two steps.
Under the assumptions of \cite{ref3} (a.o. $a$ is a prime number) the linear system $g^1_a$ is unique on the curve.
Then given some fixed linear system $g^1_a$ on the curve, the author proves the upper bound on the number of total ramification points of $g^1_a$ having Weierstrass semigroup $<a;b>$.
In case $(a;b)=1$ and $b \neq a+2$, the uniqueness of $g^1_a$ in case a curve $C$ of genus $(a-1)(b-1)/2$ has a Weierstrass point $P$ with $\WS(P) =<a;b>$ follows from results  from \cite{ref2} (see Theorem 10.1 for the relation).
However we give an independent proof inspired by \cite{ref3} but using seemingly easier arguments and not using the assumption that $a$ is a prime number.
So to prove Theorem 1, we only need to consider total ramification points on a fixed $g^1_a$.
Using more complicated computations than ours, Theorem \ref{theoremA} is proved in \cite{ref8} for the case of Galois Weierstrass points (meaning the morphism $C \rightarrow \mathbb{P}^1$ defined by $\vert aP \vert$ defines a Galois extension $\mathbb{C}(\mathbb{P}^1) \subset \mathbb{C}(C)$).

Smooth curves $C$ having a Weierstrass point $P$ with $\WS (P)=<a,b>$ in case $(a,b)=1$ are also called $C_{a,b}$ curves.
They are studied from different points of view (see e.g. \cite {ref10}, \cite {ref11}, \cite {ref12}, \cite {ref13}, \cite {ref14}).
In \cite {ref15} and \cite {ref16} the similar nodal curves are used to develop a general method to study Weierstrass points.

For lower genus cases $g<(a-1)(b-1)/2$ with $(a;b)=1$ and $b=na+r$ with $n \geq 1$ and $1 \leq r \leq a-1$ we consider the following situation.
Let $C$ be a smooth curve of genus $g$ and let $P \in C$ such that $a$ is the first non-gap of $P$, $b$ is the first non-gap of $P$ that is not a multiple of a and there are no other non-gaps between $na$ and $(n+1)a$.
We obtain sufficient conditions in terms of $\WS (P)$ implying the uniqueness of the linear system $g^1_a$ (this cannot be obtained using the results from \cite{ref2}).
In particular in case $b$ is much larger than $a$ then $g^1_a$ is unique (independent from the value of $g$).

We concentrate on points $Q$ on $C$ with $Q \neq P$ such that $aQ \in \vert aP \vert$ and we obtain the following theorem in this described situation.
\begin{theoremB} \normalfont \label{theoremB}
If $g>(a-1)(b-a+r)/2$ then $b \notin WG(Q)$.
\end{theoremB}

Therefore for large values of $b$ with respect to $a$ we obtain a lot of Weierstrass semigroups that can occur at most once.
Moreover we prove that in many cases this bound on the genus in Theorem B is sharp.
This means in case $g=(a-1)(b-a+r)/2$ then there exists a Weierstrass semigroup $H_0$ of genus $g$ containing $<a;b>$ and a curve $C$ of genus $g$ having two Weierstrass points with semigroup $H_0$.
Moreover this semigroup is unique.

In Section 2 we mention some general results.
In particular Lemma \ref{lemmaX} will be the basic lemma for obtaining the unicity of the pencil $g^1_a$.

In Section 3 we prove the main results of this paper.
It starts with a very easy Lemma \ref{lemma1} which is the basic observation of all our main results.
Assume $C$; $P$; $a$ and $b$ as before.
Using a particular plane model $\Gamma$ of the curve then it follows that equality $\WS (P)=\WS (Q)$ in case $aQ \in \vert aP \vert$ implies $Q$ corresponds to a particular type of singular point on $\Gamma$.
In particular it follows $\WS (P) \neq \WS (Q)$ in case $(a;b)=1$, $g=(a-1)(b-1)/2$ and $r \neq a-1$ (see Corollary \ref{corollary1}).
In case $b \neq a+1$ we also obtain uniqueness of $g^1_a$ in that case (Proposition \ref{proposition1}) implying Theorem A.
More general we also obtain Theorem B (Corollary \ref{corollary2}).
We also give some general statements on the uniqueness of $g^1_a$ in case $g<(a-1)(b-1)/2$  (see Proposition \ref{proposition2} and Corollaries \ref{corollary10} and \ref{corollary7}).

Using Lemma \ref{lemma1} in a more detailled manner we obtain a description for $\WS (Q)$ for all $Q \neq P$ satisfying $aQ \in \vert aP \vert$ in case $g=(a-1)(b-1)/2$ (Theorem \ref{theorem3}).
Continuing to use such arguments we obtain a list of non-gaps $\WS (P)$ needs to contain in order that there exists $Q$ satisfying $aQ \in \vert aP \vert$ with $\WS (P)=\WS (Q)$ in case $g<(a-1)(b-1)/2$ (Lemma \ref{lemma8}).
From this fact we obtain further conditions on $\WS (P)$ going below the genus bound of Theorem B and implying $\WS (P)$ occurs at most once (Corollary \ref{corollaryY}).
Moreover it also implies the genus bound in Theorem B is sharp in general (Corollary \ref{corollary6} and Lemma \ref{lemma5}) and it gives a complete description of the Weierstrass semigroup implying this sharpness (Corollary \ref{corollary5} as a corollary of Lemma \ref{lemma3}).

In Section 4 we consider some examples.
In case $a=4$ we show that for each integer $N$ there exists a genus bound $g(N)$ such that for $g > g(N)$ there are at least $N$ different Weierstrass semigroups with first non-gap equal to 4 and genus $g$ that occur at most once (remember in case $a=3$ this is not true). Those Weierstrass semigroups are very similar to each other.
Case $a=5$ illustrates that for growing values of $a$ we obtain more types of Weierstrass gap sequences that occur at  most once.
Case $a=6$ illustrates that the use of Lemma \ref{lemmaX} causes that making a formulation of Theorem B similar to Theorem A without assuming $aQ \in \vert aP \vert$ is not possible using the arguments of this paper.
Finally in case $n=1$ the genus bound in Theorem B is too small to obtain uniqueness of $g^1_a$.
Using a very rough but different argument we show how to obtain a result on Weierstrass semigroups that occur at most once in this case $n=1$ that satisfies sharpness on the genus bound which is larger than the genus bound in Theorem B.
This argument cannot be applied in case $n \geq 2$.

For two positive integers $a$ and $b$ we write $(a,b)$ to denote their largest common divisor.
In particular $(a,b)=1$ means $a$ and $b$ are mutually prime.
Remember we write $<a;b>$ to denote the subsemigroup of $\mathbb{N}$ generated by $a$ and $b$.
For a smooth projective variety $X$ we write $\omega _X$ to denote the canonical sheaf of $X$.

\section{Generalities}\label{section2}

We are going to use some models of the smooth curve $C$ on some surfaces.
We use the following terminology and facts.

Let $X$ be a smooth surface and let $D$, $E$ be two curves on $X$ without common components.
For $Q \in D \cap E$ we write $i(D.E;Q)$ to denote the intersection multiplicity of $D$ and $E$ at $Q$.
We also write $(D.E)$ to denote the intersection number of $D$ and $E$ on $X$.

Let $X$ be a smooth surface and let $\Gamma$ be an irreducible curve on $X$.
This curve $\Gamma$ has some arithmetic genus $p_a(\Gamma)$ and it can be computed by the formula $2p_a(\Gamma)-2=\Gamma . \left( \Gamma + K_X  \right)$ with $K_X$ a canonical divisor on $X$. In case $\Gamma$ is smooth then this arithmetic genus is equal to the genus of the smooth curve $\Gamma$.

Let $Q$ be a point on $\Gamma$ of multiplicity $\nu$ and let $p : X' \rightarrow X$ be the blowing-up of $X$ at $Q$.
Let $E$ be the associated exceptional divisor on $X'$ and let $\Gamma '$ be the proper transform of $\Gamma$ on $X'$.
It is well-known that $p_a(\Gamma ')=p_a(\Gamma)-\nu (\nu -1)/2$.
In case $\Gamma '$ has some singular points on $E$ one continues this process blowing-up $X'$ at the singular points of $\Gamma '$ on $E$ (such points are called infinitesimally near points on $X$ and infinitesimally near singular points of $\Gamma$) and so on untill one obtains a smooth surface $X_1$ such that for the proper transform $\Gamma _1$ of $\Gamma$ on $X_1$ all points mapping to $Q$ are smooth.
The difference $p_a(\Gamma)-p_a(\Gamma ')$ is denoted by $\delta (Q)$.

\begin{definition} \normalfont \label{definitieX}
We say $Q$ is a cusp on an irreducible curve $\Gamma \subset \mathbb{P}^2$ in case for the normalisation $C \rightarrow \Gamma$ there is only one point of $C$ mapping to $Q$  (i.e. $\Gamma$ is locally analytically irreducible at $Q$). Let $\nu$ be the multiplicity of $\Gamma$ at $Q$. There is a unique line $T$ on $\mathbb{P}^2$ containing $Q$ such that $i(T.\Gamma;Q)=\mu > \nu$. We say $Q$ is a cusp of type $(\nu ; \mu)$ on $\Gamma$.
\end{definition}

The following lemma should be well-known.
\begin{lemma} \normalfont \label{lemma4}
Let $\Gamma \subset \mathbb{P}^2$ be an irreducible plane curve and assume $Q$ is a cusp of type $(\nu; \mu)$  on $\Gamma$. In case $(\nu,\mu )=1$ and $Q$ is a cusp of type $(\nu ;\mu)$ then $\delta_Q=\frac{(\nu-1)(\mu -1)}{2}$.
\end{lemma}
\begin{proof}
Using blowings-up starting at $Q$ we obtain a sequence of singular points of $\Gamma$ infinitesimally near to $Q$ of known multiplicity as follows.
We make the sequence $(c_1; c_2; \cdots ; c_{k+1}=1)$ taking $c_1=\mu$ and $c_2=\nu$. Then $c_1=n_2c_2+c_3$ with $1 \leq c_3 \leq c_2-1$. In case $i \geq 3$ and $c_i \neq 1$ then $c_{i-1}=n_ic_i+c_{i+1}$ with $1 \leq c_{i+1} \leq c_i-1$.
This is the Euclidean algorithm to compute $(\nu,\mu)$.
Since $(\nu,\mu)=1$ one has $(c_i,c_{i+1})=1$ for all $1 \leq i \leq k$ and $c_{k+1}=1$.

Then for $1 \leq i \leq k$ there are $n_i$ singular points of multiplicity $c_i$ on the curve $\Gamma$ infinitesimally near to $Q$.
This implies
\[
\delta _Q=\sum_{i=2}^{k}n_i\frac{c_i(c_i-1)}{2} \text { .}
\]
For $2 \leq j \leq k$ let $\delta _j=\sum _{i=j}^{k}\frac{c_i(c_i-1)}{2}$.
By means of induction we show $\delta _j=\frac{(c_{j-1}-1)(c_j-1)}{2}$.
Since $\delta _Q=\delta _2$ this implies the lemma.

For $j=k$ we have $\delta_k=n_k\frac{c_k(c_k-1)}{2}$.
Also $c_{k-1}=n_kc_k+1$.
This implies $\delta_k=\frac{(c_{k-1}-1)(c_k-1)}{2}$.

Assume $3 \leq j \leq k$ and $\delta _j=\frac{(c_{j-1}-1)(c_j-1)}{2}$.
We have $\delta _{j-1}=n_{j-1}\frac {c_{j-1}(c_{j-1}-1)}{2}+\delta_j$.
We use $c_{j-2}=n_{j-1}c_{j-1}+c_j$ hence $\delta _{j-1}=\frac{(c_{j-2}-c_j)(c_{j-1}-1)}{2}+\frac{(c_{j-1}-1)(c_j-1)}{2}=\frac{(c_{j-1}-1)(c_{j-2}-1)}{2}$.
\end{proof}

Let $X$ be the surface $\mathbb{P}^1 \times \mathbb{P}^1$.
For each divisor $D$ on $X$ there exist unique integers $\alpha$ and $\beta$ such that $D$ is linearly equivalent to $\alpha \left( \mathbb{P}^1 \times \{ S \} \right)+\beta \left( \{ S \} \times \mathbb{P}^1 \right)$ (see e.g. \cite {ref9}, Chapter II, Example 6.6.1).
Such curve is said to be of type $(\alpha ;\beta )$ and we write $\vert (\alpha ;\beta ) \vert$ to denote the complete linear system of curves of type $(\alpha ;\beta )$. We write $\mathcal{O}_X(\alpha ;\beta )$ to denote the corresponding invertible sheaf.
For an irreducible curve $\Gamma$ on $X$ there exist so-called canonically adjoint curves to $\Gamma$ describing all elements of the canonical linear system on the normalisation of $\Gamma$.
Although this should be well-known we include an argument for this fact.

\begin{lemma} \normalfont \label{lemmaY}
Let $\Gamma$ be an irreducible curve of type $(\alpha ;\beta )$ on $X=\mathbb{P}^1 \times \mathbb{P}^1$ and let $C$ be the normalisation of $\Gamma$. Let $\vert K_C \vert$ be the canonical linear system on $C$. There exists a linear subsystem of $\vert (\alpha -2;\beta -2) \vert$ called the linear system of canonically adjoint curves of $\Gamma$ that has a natural bijective correspondence with $\vert K_C \vert$ using intersections.
\end{lemma}

\begin{proof}
Let $\pi : Y \rightarrow X$ be a sequence of blowings-up at some points (some of them might be infinitesimally near points) such that the proper transform of $\Gamma$ on $Y$ is smooth (so we identify  it with $C$).
It is well-known that $H^i(X;\mathcal{O}_X) \cong H^i(Y;\mathcal{O}_Y)$ for all $i \geq 0$ (see \cite {ref9}, Chapter V, Proposition 3.4).
Since $H^1(X;\mathcal{O}_X)=0$ (see \cite {ref9}, Chapter III, Exercise 5.6) one has $H^1(Y;\mathcal{O}_Y)=0$.
Canonical divisors on $X$ are of type $(-2;-2)$ (see \cite {ref9}, Chapter II, Exercise 8.20.3).
From Serre duality (we use \cite {ref9}, Chapter III, Corollary 7.7) it follows $H^2(X;\mathcal{O}_X) \cong H^0(X;\mathcal{O}_X(-2;-2))=0$ and therefore $H^2(Y;\mathcal{O}_Y)=0$ and also $H^0(Y;\omega _Y)=0$ and $H^1(Y;\omega _Y)=0$.
From \cite{ref9}, Chapter II, Proposition 8.20 we know $\omega _C\cong \omega_Y \otimes \mathcal{O}_Y \otimes \mathcal{O}_C$.
Tensoring the exact sequence
\[
0 \rightarrow \mathcal{O}_Y(-C) \rightarrow \mathcal{O}_Y \rightarrow \mathcal{O}_C \rightarrow 0
\]
with $\mathcal{O}_Y(C) \otimes \omega _Y$ gives rise to the exact sequence
\[
0 \rightarrow \omega _Y \rightarrow \omega _Y \otimes \mathcal{O}_Y(C) \rightarrow \omega _C \rightarrow 0
\]
Using the exact cohomology sequence we obtain an isomorphism $H^0(Y;\omega _Y\otimes \mathcal{O}_Y(C))\rightarrow H^0(C;\omega _C)$.
The images on X of elements of the complete linear system associated to $\omega_Y \otimes \mathcal{O}_Y(C)$ are the canonically adjoint curves of $\Gamma$.
From the construction it follows they are contained in $\vert (\alpha -2;\beta -2) \vert$ (this follows from an explicit description of $\omega _Y$ using the blowings-up (see \cite {ref9}, Chapter V, proposition 3.3)) and from the proof it follows they are in bijective correspondence to effective canonical divisors on $C$.
\end{proof}

Unicity of a linear system $g^1_a$ will be a consequence of the following lemma.

\begin{lemma} \normalfont \label{lemmaX}
Let $C$ be a smooth curve of genus $g$ and let $P$ be a point of $C$ with first non-gap equal to $a$.
Assume $C$ has a base point free linear system $g^1_a$ different from $\vert aP \vert$.
There there exists a divisor $e<a$ of $a$ (it might be $1$) such that each integer $\left( \frac{a}{e}-1 \right)a+ie$ with $i \in \mathbb{Z}_{\geq 1}$ is a non-gap of $P$.
\end{lemma}

\begin{proof} Let $f_1 : C \rightarrow \mathbb{P}^1$ be a morphism corresponding to $\vert aP \vert$ and let $f_2 : C \rightarrow \mathbb{P}^1$ be a morphism corresponding to $g^1_a$.
Consider the morphism $f=(f_1;f_2) : C \rightarrow \mathbb{P}^2 \times \mathbb{P}^1$ and let $\Gamma$ be the image of $f$.
Let $C'$ be the normalisation of $\Gamma$, then $f$ factorizes through a finite morphism $h : C \rightarrow C'$ of some degree $e<a$ dividing $a$ ($e$ might be equal to 1).
The rulings of $\mathbb{P}^1 \times \mathbb{P}^1$ imply base point free linear systems $g_1$ and $g_2$ on $C'$ such that $h^{-1} (g_1)=\vert aP \vert$; $h^{-1} (g_2)=g^1_a$.
In particular for $P'=h(P) \in C'$ one has $\frac{a}{e}P' \in g_1$.
The canonically adjoint curves of $\Gamma$ give rise to a linear subsystem of $\vert (\frac{a}{e}-2; \frac{a}{e}-2) \vert$ and they correspond bijectively with effective canonical divisors on $C'$

Let $E$ be a general element of $\vert \frac{a}{e}P' \vert$.
An effective canonical divisor on $C'$ containing $E$ corresponds to some curve $\gamma$ in $\vert (\frac{a}{e}-2;\frac{a}{e}-2) \vert$ containing $E$.
This divisor $E$ consists of $\frac{a}{e}$ different points on some line $l$ belonging to $\vert (1;0) \vert$.
Since the intersection number $\left( (\frac{a}{e}-2;\frac{a}{e}-2).(1;0) \right)=\frac{a}{e}-2<\frac{a}{e}$ it follows $l \subset \gamma$.
This implies there is no canonically adjoint curve of $\Gamma$ containing $\frac{a}{e}-1$ general elements of $\vert \frac{a}{e}P' \vert$.
Therefore no effective canonical divisor of $C'$ contains $\frac{a}{e}-1$ general elements of $\vert \frac{a}{e}P' \vert$. It follows $(\frac{a}{e} -1)\frac{a}{e}P'$ is a non-special divisor on $C'$.
This implies for each $i \in \mathbb{Z}_{\geq 1}$ the integer $(\frac{a}{e}-1)\frac{a}{e}+i$ is a non-gap of $P'$.
Using the inverse image under the morphism $f_1$ one obtains $(\frac{a}{e}-1)a+ie$ is a non-gap of $P$.
\end{proof}

\section{Proofs}\label{section3}

This easy lemma having a trivial proof is the basic lemma for all main results in this paper.
\begin{lemma} \normalfont \label{lemma1}
Let $C$ be a smooth curve, $P \in C$ and $a,b \in \mathbb{Z}_{\geq 1}$ with $b=na+r$ with $r,n \in \mathbb{Z}$ satisfying $0< r <a$ and $n \geq 1$. 
Assume $a,b \in \WS (P)$.
Assume $Q \in C$ with $Q \neq P$ and $aQ \in \vert aP \vert$.
Let $\mu \in \mathbb{Z}_{\geq 1}$ with $0 < \mu <a$ and assume $\vert bP-\mu Q \vert$ does not have $Q$ as a base point.
Then $an+(a-\mu) \in \WS (Q)$.
\end{lemma}

\begin{proof}
Since $aQ \in \vert aP \vert$ it follows that
\[
D_0:=(n-1)aQ+rP+(a-\mu)Q \in \vert bP-\mu Q \vert \text { .}
\]
Since $Q$ is not a base point of $\vert bP-\mu Q \vert$ it follows $Q$ is not a base point of $\vert bP-\mu Q+(a-r)P \vert$.
However $D_0+(a-r)P=(n-1)aQ+aP+(a-\mu) Q$ and again using $aQ \in \vert aP \vert$ we obtain
\[
D_1 := naQ+(a-\mu)Q \in \vert bP-\mu Q+(a-r)P \vert \text { .}
\]
This implies $Q$ is not a base point of $\vert (na+(a-\mu))Q \vert$ hence $\vert (na+(a-\mu))Q \vert$ is base point free.
This implies $na+(a-\mu)$ is a non-gap of $Q$.
\end{proof}

From now on in this paper we make the following assumptions.
$C$ is a smooth curve of genus $g$ and $P$ is a smooth point of $C$. We assume $\vert aP \vert$ is a base point free $g^1_a$ (i.e. $a$ is the first non-gap of $P$).
Let $n \in \mathbb{Z}_{\geq 1}$ such that $\dim \vert naP \vert=n$ while $\dim \vert (n+1)aP \vert >n+1$.
Such $n$ exists and it is unique.
This means the first non-gap $b$ of $P$ that is not a multiple of $a$ is of type $b=an+r$ with $0<r<a$.

\begin{lemma} \normalfont \label{lemma2}
Let $Q \in C$ with $Q \neq P$ and $aQ \in \vert aP \vert$.
There is a unique integer $\mu$ satisfying $0< \mu <a$ such that $\vert bP-\mu Q \vert$ does not have $Q$ as a base point.
\end{lemma}
\begin{proof}
Since $aQ \in \vert aP \vert$ one has $(b-a)P \in \vert bP-aQ \vert$, hence $Q$ is not a fixed point of $\vert bP-aQ \vert$.
From the definition of $a$ and $b$ it follows $\dim \vert bP-aQ \vert=n-1=\dim \vert bP \vert -2$.
Assume $\vert bP-Q \vert$ contains $Q$ as a base point with multiplicity $\nu$ ($\nu$ can be equal to 0).
Then $\vert bP-(\nu+1)Q \vert$ does not contain $Q$ as a base point and $\dim \vert bP-(\nu +1)Q \vert=\dim \vert bP \vert-1=n=\dim \vert bP-aQ \vert +1$.
In particular $\nu +1<a$.
This implies the existence of an integer $\mu$ satisfying $0<\mu <a$ such that $\vert bP- \mu Q\vert$ does not contain $Q$ as a base point (taking $\mu = \nu +1$).

In case there exists an integer $\mu '\neq \mu$ with $0< \mu' <a$ such that $\vert bP-\mu' Q\vert$ does not contain $Q$ as a base point, then $\mu '>\mu$ and we find $\dim \vert bP-\mu 'Q\vert=n-1$ and $\dim \vert bP-aQ \vert <n-1$, a contradiction.
\end{proof}

Under the assumptions of Lemma \ref{lemma2} it follows from Lemma \ref {lemma1} that $an+(a-\mu)$ is a non-gap of $Q$.
In case $\dim \vert (n+1)aP \vert =n+2$ then there is a unique non-gap of $Q$ between $an$ and $a(n+1)$.
So we obtain the following conclusion.

\begin{corollary} \normalfont \label{corollary8}
Assume $\dim \vert (n+1)aP \vert =n+2$ and $Q \in C$ with $Q \neq P$ and $aQ \in \vert aP \vert$.
If $\WS (P)=\WS (Q)$ then $\mu = a-r$ is the unique integer $0<\mu <a$ such that $\vert bP - \mu Q \vert$ does not have $Q$ as a base point.
\end{corollary}

In case the linear system $\vert bP \vert$ is simple then we can give a geometric meaning to the number $\mu$ occuring in Lemma \ref{lemma2} and Corollary \ref{corollary8} using some specific plane model $\Gamma \subset \mathbb{P}^2$ of $C$.
As in \cite{ref1} we construct a simple base point free linear system $g^2_b$ on $C$ as follows.
Choose $D \in \vert bP \vert$ general, in particular $P \notin D$.
Inside the projective space $\vert bP \vert$ take the linear span of the line $\vert aP \vert +(b-a)P$ and $D$ (denoted by $<\vert aP \vert +(b-a)P;D>$).
This linear systems $g^2_b$ defines a morphism from $C$ to $\mathbb{P}^2$ and the image $\Gamma \subset \mathbb{P}^2$ is a plane curve of degree $b$ birationally equivalent to $C$.
(We write $\phi : C \rightarrow \Gamma$ to denote the normalization.)
The image $\phi (P)$ is a cusp of $\Gamma$ of type $(b-a;b)$.
This singularity causes that the genus of the curve $C$ is at most $((b-1)(a-1)+1-(a,b))/2$ (see the computation in Section 1 of \cite{ref1}).
It should be mentioned that it is proved in \cite {ref1}, Section 3 that there exist such curves $C$ for all $g \leq ((a-1)(b-1)/2+1-(a,b))/2$ (see also \cite {ref15}, Section 3 in case $(a,b)=1$).
From now on we assume $\vert bP \vert$ is simple and $\Gamma \subset \mathbb{P}^2$ is such a plane model of $C$.

Assume $Q \in C$ with $Q \neq P$ and $aQ \in \vert aP \vert$.
Clearly $\phi (Q) \neq \phi (P)$ (since $bP \in g^2_b$).
Let $L_Q$ be the line in $\mathbb{P}^2$ connecting $\phi (P)$ and $\phi (Q)$.
Since the pencil of lines on $\mathbb{P}^2$ through $\phi (P)$ induces $\vert aP \vert$ on $C$ (because $(b-a)P+\vert aP \vert \subset g^2_b$) it follows $i(L_Q.\Gamma;\phi (Q))=a$.

Let $\mu$ be the multiplicity of $\Gamma$ at $\phi (Q)$. We already know $\mu \leq a$.
In case $\mu = a$ then it would imply $\vert bP-aQ \vert$ is base point free.
Since $\vert bP-aQ\vert = \vert (b-a)P \vert$ this would contradict the meaning of the integers $a$ and $b$.
It follows $1 \leq \mu \leq a-1$ and $\phi (Q)$ is a cusp of type $(\mu; a)$ of $\Gamma$.
The pencil of lines in $\mathbb{P}^2$ containing $\phi (Q)$ induces a base point free linear system on $C$ contained in $\vert bP-\mu Q \vert$.
Therefore the multiplicity of $\phi (Q)$ on $\Gamma$ is the integer $0 < \mu < a$ mentioned in Lemma \ref{lemma2} and Corollary \ref{corollary8}.
Using this plane model $\Gamma$ of $C$ we obtain the following conclusion.

\begin{corollary} \normalfont \label{corollary9}
Assume $\vert bP \vert$ is simple, $\dim \vert (n+1)aP \vert = n+2$ and $Q \in C$ with $Q \neq P$ and $aQ \in \vert aP \vert$.
If $\WS (P)=\WS (Q)$ then $\phi (Q)$ is a cusp of $\Gamma$ of type $(a-r;a)$.
\end{corollary}

\begin{corollary} \normalfont \label{corollary1}
Assume $\vert bP \vert$ is simple and $Q \in C$ with $Q \neq P$ and $aQ \in \vert aP \vert$. 
In case $g=((a-1)(b-1)+1-(a;b))/2$ and $r\neq a-1$ then $\WS (P) \neq \WS (Q)$.
\end{corollary}
\begin{proof}
From the condition $g=((a-1)(b-1)+1-(a,b))/2$ it follows $\dim \vert (n+1)aP \vert = n+2$.
Let $\Gamma$ be a plane model of $C$ as described before.
Then all points on $\Gamma$ different from $\phi (P)$ are smooth.
This implies $\phi (Q)$ is a cusp of $\Gamma$ of type $(1;a)$ and $na+(a-1)$ is the non-gap of $Q$ between $an$ and $a(n+1)$.
Therefore the Weierstrass semigroups of $Q$ and $P$ can be equal only in case $r=a-1$.
\end{proof}

From now on we assume $(a,b)=1$ with $a<b$.
In this case $\vert bP \vert$ is simple.
Write $b=na+r$ with $n \geq 1$ and $0< r <a$.
The equation of the plane model $\Gamma$ can be reduced to some canonical form (see e.g. \cite {ref15} Lemma 6.2). 
In case $g=(a-1)(b-1)/2$ then $\phi (P)$ is the only singular point on $\Gamma$ and such curves are the so-called $C_{a,b}$ curves.
In this case one has $\WS (P)=<a,b>$.

In case $g=(a-1)(b-1)/2$ and $b=a+1$ then $C=\Gamma$ is a smooth plane curve of degree $a+1$ defined by the linear system $\vert bP \vert$ (hence $P$ is a total inflection point of this smooth plane curve $\Gamma$).
For each point $Q$ on $C$ the linear system $\vert bP-Q \vert$ is a base point free linear system $g^1_a$ on $C$.
In case $Q$ is also a total inflection point of $\Gamma$ (i.e. $bQ \in \vert bP \vert$) then also $\WS (Q)=<a,b>$.
In that way $C$ can have many Weierstrass points having Weierstrass semigroup equal to $<a,a+1>$.

Now we are going to prove that in case $g=(a-1)(b-1)/2$ and $b \neq a+1$ then the linear system $\vert aP \vert$ is the unique linear system $g^1_a$ on $C$ without base points.
This implies that any point $Q$ on $C$ having $\WS (Q)=<a,b>$ satisfies $aQ \in \vert aP \vert$.

\begin{proposition} \normalfont \label{proposition1}
Let $C$ be a curve of genus $g=(a-1)(b-1)/2$ with $(a,b)=1$ and assume $C$ has a Weierstrass point with Weierstrass semigroup $<a;b>$.
In case $b\neq a+1$ then $C$ has a unique linear system $g^1_a$.
\end{proposition}
\begin{proof}
Assume $C$ has more than one linear system $g^1_a$.
From Lemma \ref{lemmaX} it follows that there exists a divisor $e$ of $a$ different from $a$ such that for all integers $i\geq 1$ the integer $(\frac{a}{e}-1)a+ie$ is a non-gap of $P$.
By assumption those integers belong to $<a;b>$ hence each one of them can be written as $xa+yb$ for some non-negative integers $x$ and $y$.
Since $(a,b)=1$ and $e$ divides $a$ it follows $e$ divides $y$.
Therefore for each integer $1 \leq i \leq \frac{a}{e}-1$ there is a pair of integers $(x_i;y_i)$ with $x_i \geq 0$ and $0<y_i<\frac{a}{e}$ such that $(\frac{a}{e}-1)a+ie=x_ia+y_ieb$.
This implies there is an integer $k_i$ such that $y_ib=i+k_i\frac{a}{e}$.
In case $y_i=y_{i'}$ the this implies $(i-i')=(k_i-k_{i'})\frac{a}{e}$ and therefore $i=i'$.
This implies there exists some $1 \leq i \leq \frac{a}{e}-1$ such that $y_i=\frac{a}{e}-1$ and therefore 
\[
x_ia+(\frac{a}{e}-1)eb=(\frac{a}{e}-1)a+ie<\frac{a}{e}a \text { .}
\]
In case $e \geq 2$ one has $a-e = e(\frac{a}{e}-1)\geq 2(\frac{a}{e}-1)\geq \frac{a}{e}$ (since $e \neq a$ one has $\frac{a}{e} \geq 2$).
Since $b>a$ this implies $(a-e)b>\frac{a}{e}a$, a contradiction.
In case $e=1$ one obtains $(a-1)b<a^2$.
This is a contradiction in case $b \geq a+2$.
\end{proof}

From Proposition \ref{proposition1} and Corollary \ref{corollary1} we obtain Theorem A from the introduction.

\begin{theorem} \normalfont \label{theorem1}
Assume $g=(a-1)(b-1)/2$ and assume $b\neq a+1$ and $r<a-1$ then there is no point $Q \neq P$ such that $\WS (Q)=<a;b>$.
\end{theorem}
\begin{proof}
The condition of $<a;b>$ being the Weierstrass semigroup of $P$ is equivalent to $g=(a-1)(b-1)/2$ in case $a$ and $b$ are non-gaps of $P$ with $(a;b)=1$.
From Proposition \ref{proposition1} it follows $C$ has a unique $g^1_a$.
So in case $Q$ would have the same Weierstrass semigroup as $P$ then $aQ \in \vert aP \vert$.
However Corollary \ref{corollary1} implies that is this case $\WS (Q)\neq \WS (P)$.
\end{proof}

Assume $Q$ is a cusp of type $(\mu; a)$ on the plane model $\Gamma$ with $(a,\mu )=1$.
In this case $\phi ^{-1}(Q) \subset C$ consists of exactly one point also denoted by $Q$.
From Lemma \ref{lemma4} it follows the genus of $C$ is at most $(a-1)(b-\mu )/2$.
Note that in case $\mu = a-r$ then from $(a,b)=1$ it follows $(a,\mu )=1$.
From Corollary \ref{corollary9} we obtain Theorem B from the introduction.

\begin{corollary} \normalfont \label{corollary2}
Assume $\dim \vert (n+1)aP \vert=n+2$ and let $Q \in C$ with $Q \neq P$ and $aQ \in \vert aP \vert$.
In case $g>(a-1)(b-a+r)/2$ then $\WS (Q) \neq \WS (P)$.
\end{corollary}

In case $C$ has an linear system $g^1_a$ different from $\vert aP \vert$ it follows from Lemma \ref{lemmaX} there exists a divisor $e$ of $a$ different from $a$ such that $\left( \frac{a}{e}-1 \right)\frac {a}{e}+e$ is a non-gap of $P$.
As a rough estimate this implies there is a non-gap of $P$ not being a multiple of $a$ having value at most $\left( \frac{a}{2} \right)^2$.
Because of the meaning of $b$ this is impossible in case $b>\left( \frac{a}{2} \right)^2$.
Therefore in this case Corollary \ref{corollary2} implies the following statement on uniqueness of Weierstrass semigroups.

\begin{corollary} \normalfont \label{corollary10}
Assume $\dim \vert (n+1)aP \vert = n+2$ , $b>\left( \frac{a}{2} \right)^2$ and $g>(a-1)(b-a+r)/2$.
Then $\WS (P)$ occurs at most once.
\end{corollary}

The estimate used in Corollary \ref{corollary10} is very rough.
Using some (still rough) estimates on the number of non-gaps we obtain the following condition implying uniqueness of the $g^1_a$ in case $g<(a-1)(b-1)/2$.

\begin{proposition} \normalfont \label{proposition2}
Assume for each divisor $e$ of $a$ different from $a$ (but including 1) one has
\[
g>\frac{(a-1)(b-1)}{2}-\frac{(an-ne-\frac{a}{e}+2)(an-ne-\frac{a}{e}+1)}{2ne}
\]
then $C$ has a unique base point free $g^1_a$.
\end{proposition}

\begin{proof}
Let $x$ be some integer at least 1.
From divisibility arguments as used in the proof of Proposition \ref{proposition1}, in case $(a;b)=1$ the number of elements of the type $xa+ie$ for some integer $1 \leq i \leq \frac{a}{e}-1$ inside $<a;b>$ is at most $[\frac{x}{ne}]$ (here $b=na+r$ with $0<r<a$).

Assume $C$ has a base point free $g^1_a$ different from $\vert aP \vert$. In case $x \geq \frac{a}{e}-1$ it follows from Lemma \ref{lemmaX} that $xa+ie$ is a non-gap of $P$ for each integer $1 \leq i \leq \frac{a}{e}-1$.
Therefore there are at least $(\frac{a}{e}-1)-[\frac{x}{ne}]\geq (\frac{a}{e}-1)-\frac{x}{ne}$ non-gaps of $P$ between $xa$ and $(x+1)a$ outside $<a;b>$.
Summing up over different values of $x$ we obtain at least $\frac{(an-ne-\frac{a}{e}+2)(an-ne-\frac{a}{e}+1)}{2ne}$ non-gaps of $P$ outside $<a;b>$.
\end{proof}

In case $a$ is a prime number we only have to consider the case $e=1$ in the statement of Proposition \ref{proposition2}.
In particular we obtain the following statement concerning uniqueness of Weierstrass semigroups.

\begin{corollary} \normalfont \label{corollary7}
Let $a$ be a prime number and assume $b>3a$ is an integer not divisible by $a$.
Write $b=na+r$ with $1 \leq r \leq a-1$.
Let $H$ be a Weierstrass semigroup containing $<a;b>$  having no non-gap outside of $<a;b>$ smaller than $(n+1)a$ and having genus $g>\frac{(a-1)(b-a+r)}{2}$.
Then $H$ occurs at most once.
\end{corollary}
\begin{proof}
From Corollary \ref{corollary2} it follows that in case there exists a smooth curve of genus $g$ having two different Weierstrass points $P$ and $Q$ with Weierstrass semigroup equal to $H$ then $aP$ and $aQ$ are not linearly equivalent.
In particular $C$ has a base point free linear system $g^1_a$ different from $\vert aP \vert$.
From Proposition \ref{proposition2} we know this implies
\[
g\leq \frac{(a-1)(b-1)}{2}-\frac{(an-n-a+2)(an-n-a+1}{2n} \text { .}
\]
Since $n \geq 3$ this implies $g\leq \frac{(a-1)(b-1)}{2}-\frac{(2a-1)(2a-2)}{6}$.
Since $\frac{(a-1)(a-2)}{2} < \frac{(2a-1)(a-1)}{3}$  we obtain a contradiction.
\end{proof}

In Section 4 we illustrate that using Lemma \ref{lemmaX} gives rise to better uniqueness statements for the linear system $g^1_a$  in case of explicit examples than using Proposition \ref{proposition2}.
It should be noted that the results of \cite {ref15} imply that many of those Weierstrass semigroups really occur as Weierstrass semigroups of points on certain curves.

Lemma \ref{lemma1} can also be used to determine the Weierstrass semigroups of the points $Q \neq P$  satisfying $aQ \in \vert aP \vert$ in some situations.

\begin{theorem} \normalfont \label{theorem3}
Assume $g=(a-1)(b-1)/2$ and let $Q \in C$ with $Q \neq P$ and $aQ \in \vert aP \vert$.
For $t \in \mathbb{Z}{\geq 1}$ let $s$ be the number of non-gaps $e$ of $P$ satisfying $ta<e<(t+1)a$.
Then $(t+1)a-i$ with $1 \leq i \leq s$ are the non-gaps of $Q$ satisfying $ta<e<(t+1)a$.
\end{theorem}

\begin{proof}
In case $s=a-1$ the theorem is trivially true, so we assume $s<a-1$.
The genus of $C$ implies $\WS (P)=<a;b>$.
Therefore the integer $s$ associated to $t$  is defined by the inequalities $sb<(t+1)a$ and $(s+1)b>(t+1)a$.
Define $\epsilon \in \mathbb{Z}_{\geq 0}$ such that $at < sb+\epsilon a<a(t+1)$.
Hence $sb+\epsilon a$ is a non-gap $e$ of $P$ satisfying $at < e < a(t+1)$.
Because of the genus of $C$ it follows that $Q$ corresponds to a smooth point on the plane model $\Gamma$ of $C$, hence $\vert bP-Q \vert$ does not have $Q$ as a base point.

For each integer $0 \leq i \leq s$ one has
\[
i(bP-Q)+\left( (s-i)b+\epsilon a   \right)P\in \vert (sb+\epsilon a)P-iQ \vert \text { .}
\]
This implies $\vert (sb+\epsilon a)P -iQ \vert$ does not contain $Q$ as a fixed point.
From Lemma \ref{lemma1} it follows $at+(a-i)$ is a non-gap of $Q$ for $1 \leq i \leq s$.
\end{proof}

In \cite{ref8}, Lemma 2.7 the authors also obtain the statement of Theorem \ref{theorem3} assuming $P$ is a Galois Weierstrass point.
The statement is formulated in a diiferent way and it needs some computations to show both descriptions of the Weierstrass semigroup are the same.

\begin{corollary} \normalfont \label{corollary3}
Assume $g=(a-1)(b-1)/2$ and $r=a-1$. Then for all $Q \in C$ with $aQ \in \vert aP \vert$ one has $\WS (Q)=<a;b>$.
\end{corollary}

Now we consider the case $g \leq (a-1)(b-a+r)/2$ to obtain a refinement of Corollary \ref{corollary2} and a generalisation of Theorem \ref{theorem3}.
In many cases it also implies sharpness of Corollary \ref{corollary2}.
Assume $P$ as before and $Q \neq P$ such that $aQ \in \vert aP \vert$ and $Q$ is a cusp of type $(\mu =a-r ;a)$ on the plane model $\Gamma$.
We assume $(a,b)=1$, hence $(a,\mu )=1$.
For $1 \leq m \leq \mu -1$ define the integer $n(m)$ such that $(n(m)-1)\mu < ma < n(m)\mu$.
Note that $m \leq \mu -1$ implies $n(m)<a$.

\begin{lemma} \noindent \label{lemma8}
For $n(m) \leq i \leq a-1$ one has $ib-ma$ is a non-gap of $Q$.
\end{lemma}
\begin{proof}
Since $\vert bP-\mu Q \vert$ does not have $Q$ as a base point it follows $\vert n(m)bP-n(m)\mu Q\vert = \vert (n(m)b-ma)P-(n(m)\mu-ma)Q \vert$ does not have $Q$ as a base point.
One has $n(m)b \equiv -n(m)\mu \mod a$.
This implies $\left( n(m)b-2ma+n(m)\mu   \right)P$ is linearly equivalent to $\left(  n(m)b-2ma+n(m)\mu \right)Q$.
Since $\vert \left( n(m)b-ma  \right)P-\left( n(m)\mu -ma  \right)Q+\left( n(m)\mu -ma  \right)P \vert$ does not have $Q$ as a base point and contains $\left( n(m)b-ma \right)Q$  it follows $\vert (n(m)b-ma)Q  \vert$ has no base point, hence $n(m)b-ma \in \WS (Q)$.

Since $b \in \WS (Q)$ it follows that for all integers $n(m) \leq i \leq a-1$ the integer $ib-ma \in \WS (Q)$.
\end{proof}

For all integers $1 \leq i <a$ one has $ib$ is the smallest integer $x$ in $<a;b>$ satisfying $x \equiv ib \mod a$.
Therefore the integers $ib-ma$ with $n(m) \leq i \leq a-1$ do not belong to $<a;b>$.
Varying $1 \leq m \leq \mu -1$ we obtain $\sum_{m=1}^{\mu -1} (a-n(m))$ non-gaps of $Q$ not belonging to $<a;b>$.
We call them the \emph{trivial new non-gaps} associated to a cusp of type $(\mu ; a)$ on the plane model $\Gamma$.

\begin{corollary} \normalfont \label{corollaryY}
Assume $\dim \vert naP \vert=n+2$ and let $Q \in C$ with $Q \neq P$ and $aQ \in \vert aP \vert$.
In case the Weierstrass semigroup of $P$ does not contain the list of trivial new non-gaps asssociated to a cusp of type $(a-r,a)$ on the plane model $\Gamma$.
Then $\WS (P) \neq \WS (Q)$.
\end{corollary}

Combining Corollary \ref{corollaryY} with Proposition\ref{proposition2} we obtain the following statement.

\begin{corollary} \normalfont \label{corollary4}
Assume $\dim \vert naP \vert = n+2$ and assume the Weierstrass semigroup of $P$ does not contain the list of trivial non-gaps asssociated to a cusp of type $(a-r,a)$ on the plane model $\Gamma$.
Assume for each divisor $e$ of $a$ different from $a$ one has
\[
g>\frac{(a-1)(b-1)}{2}-\frac{(an-ne-\frac{a}{e}+2)(an-ne-\frac{a}{e}+1)}{2}
\]
then for each point $Q \in C$ with $Q  \neq P$ one has $\WS (Q) \neq \WS (P)$.
\end{corollary}

In case $g=(a-1)(b-a+r)/2$ Lemma \ref{lemma8}  implies that the Weierstrass semigroup is completely determined in case $Q$ is an cusp of $\Gamma$ of type $(a-r;a)$.
This follows from the calculations made in the following lemma.
\begin{lemma} \normalfont \label{lemma3}
The number of trivial new non-gaps associated to a cusp of type $(\mu,a)$ on the plane model $\Gamma$ is equal to $(a-1)(\mu -1)/2$
\end{lemma}
\begin{proof}
By definition for each $1 \leq m \leq \mu -1$ one has $n(m)\leq a-1$ and $n(m) \geq 2$.
Also for $2 \leq k \leq a-1$ there is at most one integer $1 \leq m \leq \mu-1$ with $n(m)=k$.
In case such $m$ with $n(m)=k$ exists we define $x(k)=1$, otherwise $x(k)=0$.
Each integer $1 \leq k \leq a-1$ gives rise to $x(k)(a-k)$ trivial new non-gaps.
So the number of trivial new non-gaps can be written as $\sum_{k=2}^{a-1}x(k)(a-k)$.

For $1 \leq m \leq \mu -1$ write $ma=(n(m)-1)\mu +\epsilon$ with $0 < \epsilon < \mu$.
Then $(\mu - m)a=(a-n(m))\mu +(\mu - \epsilon)$ implying $n(\mu -m)=a-n(m)+1$.

In case $a$ is even and $2 \leq k \leq \frac{a-2}{2}$ it implies $x(k)=x(a-k+1)$ while $\frac{a}{2} \leq a-k+1 \leq a-1$.
So the number of trivial new non-gaps is equal to
\[
\sum_{k=2}^{(a-2)/ 2}x(k)(a-k+a-(a-k+1))=(a-1)\sum_{k=2}^{(a-2)/ 2}x(k) \text { .}
\]
On the other hand $2\sum_{k=2}^{(a-2)/2}x(k)=\mu -1$ and we obtain that the number of trivial new non-gaps is equal to $(a-1)(\mu -1)/2$ (note that $\mu$ is odd in case $a$ is even since $(a,\mu)=1$).

In case $a$ is odd and $2 \leq k \leq \frac{a+1}{2}$ one has $\frac{a+1}{2} \leq a-k+1 \leq a-1$.
In case $x\left( \frac{a+1}{2}  \right)=0$ we conclude as before.
In case $x\left(\frac  {a+1}{2} \right)=1$ the number of trivial new non-gaps is equal to
\[
\sum_{k=1}^{(a-1)/2}x(k)(a-1)+\left( \frac{a-1}{2}  \right) \text { .}
\]
On the other hand $2\sum_{k=1}^{(a-1)/2}x(k)+1=\mu -1$ and again we obtain again the number of trivial new non-gaps is equal to $(a-1)(\mu -1)/2$.
\end{proof}

\begin{corollary} \normalfont \label{corollary5}
Assume $\dim \vert naP \vert = n+2$ and assume $g=(a-1)(b-a+r)/2$.
Let $Q$ is a point on $C$ with $aQ \in \vert aP \vert$ and assume $Q$ corresponds to a cusp of type $(a-r ; a)$ on the plane model $\Gamma$.
Then $\WS (Q)$ is equal to the union $S$ of $<a;b>$ and the set of trivial new non-gaps.
\end{corollary}
\begin{proof}
From Lemma \ref{lemma8} it follows that $\WS (Q)$ contains $S$.
From the equality of the numbers obtained in Lemma \ref{lemma3} and Lemma \ref{lemma4} one finds that $\mathbb{N} \setminus S$ consists of exactly $g$ elements.
Therefore $S$ is the Weierstrass semigroup of $Q$.
\end{proof}

We are now able to prove sharpness of Corollary \ref{corollary2} in a lot of cases.

\begin{corollary} \normalfont \label{corollary6}
Same assumptions as in Corollary \ref{corollary5}. 
Assume the curve $C$ has a Weierstrass point $Q \neq P$ with non-gaps $a$ and $b$ with $ aQ \in \vert aP \vert$. 
Then $\WS (P)=\WS (Q)$ and they are both equal to the union of $<a;b>$ and the set of trivial new non-gaps.
\end{corollary}
\begin{proof}
One can make a plane model once using $P$ and once using $Q$ applying Corollary \ref{corollary5} in both cases.
\end{proof}

\begin{lemma} \normalfont \label{lemma5}
There exists a plane curve of degree $b$ having a cusp of type $(b-a;a)$ and a cusp of type $(\mu; a)$ and no other singularities.
\end{lemma}
\begin{proof}
On $\mathbb{P}^1$ choose two different points $P_0$ and $Q_0$.
Take $g^1_a=<aP_0;aQ_0>$ and choose a general effective divisor $E$ of degree $b- \mu$ on $\mathbb{P}^1$.
Take $g^2_b=<(b-a)P_0+<aP_0;aQ_0>;\mu Q+E >$.
This gives rise to a plane curve $\Gamma_0$ of degree $b$ such that $P_0$ defines a cusp of type $(b-a;a)$ and $Q_0$ defines a cusp of type $(\mu ; a)$.
As in \cite{ref1} Section 3 one can prove that all other singularities of $\Gamma_0$ are ordinary nodes.
Using Tannenbaum's result as in loc. cit. those nodes can be smoothed in a family of plane curves obtaining a plane curve $\Gamma$ of degree $b$ having a cusp of type $(b-a;a)$ and a cusp of type $(\mu ; a)$ and no other singularities.
\end{proof}

The normalisation $C$ of the plane curve obtained in Lemma \ref{lemma5} is a smooth curve of genus $g=(a-1)(b- \mu)/2$.
The point $P$ corresponding to the cusp of type $(b-a;a)$ has non-gaps $a$ and $b$.
The point $Q$ corresponding to the cusp of type $(\mu ; a)$ also has non-gaps $a$ and $b$.
In case $\dim \vert naP \vert =n+2$ then from Corollary \ref{corollary6} it follows both points have Weierstrass gap sequence equal to the union of $<a;b>$ and the set of trivial new non-gaps.
In case $\dim \vert naP \vert >n+2$ then there is a non-gap of $P$ between $an$ and $a(n+1)$ different from $b$.
This implies the existence of non-gaps not contained in $<a,b>$.
In case the number of those new non-gaps is larger than $(a-1)(\mu -1)/2$ this gives a contradiction.
In such cases Corollary \ref{corollary2} is sharp.
In case $b$ is sufficiently large with respect to $a$ and some integer $b'$ with $b <b' <(n+1)a$ is also a non-gap then the number of non-gaps not contained in $<a;b>$ is indeed larger than $(a-1)(\mu -1)/2$ and we obtain sharpness in Corollary \ref{corollary2}.

\section{Examples} \label{section4}

\begin{example} \normalfont \label{example1}
Assume $C$ is a smooth curve of genus $g$ and $P$ is a Weierstrass point on $C$ with first non-gap equal to 4.
Note that all Weierstrass semigroups with first non-gap equal to 4 occur as Weierstrass semigroup of some point on some smooth curve (see \cite{ref17}).
Let $4n+1$ with $n \geq 2$ be a non-gap of $P$ (in case $a=4$ this is the only possibility for $b$ implying the existence of Weierstrass semigroups that occur at most once in this paper).

In Lemma \ref{lemmaX} we have to consider the possibilities $e=1$ and $e=2$.
In case $e=1$ then all integers at least 12 are non-gaps of $P$.
In particular $g\leq 8$ in case $n=2$ and $g \leq 9$ in case $n\geq 3$.
In case $e=2$ then for all integers $m \geq 2$, $3m$ a non-gap of $P$.
This implies $g \leq 2n+2$.
In case $4n-3$ would be a non-gap of $P$ then $g \leq 6n-6$ and in case $4n+2$ or $4n+3$ also would be some non-gap of $P$ then one concludes $g \leq 4n$.
In particular in case $g>6n-6$ then $\dim \vert (n+1)4P \vert = n+2$ and the same conclusion holds in case $g>4n$ and $4n-3$ is a gap of $P$.
In case $g>4n$ then also $e=1$ cannot occur in Lemma \ref{lemmaX}, therefore $\vert 4P \vert$ is the unique $g^1_4$ in that case.

In case $g>4n$, one of the three integers $8n-2$; $12n-5$ and $12n-1$ is a gap of $P$ and $4n-3$ is a gap of $P$, then for all $Q \in C$ with $Q \neq P$ one has $\WS (P) \neq \WS (Q)$.
In case $g=6n-m$ with $0\leq m \leq 2n-1$ it implies $<4;4n+1>\cup \{ 4i+3 : 3n-m \leq i \leq 3n-1\}$ is a Weierstrass semigroup that occurs at most once.
The only other type of Weierstrass semigroup occuring at most once as a corollary of the results in this paper has genus $g=6n-2$ and is equal to $<4;4n+1>\cup \{ 8n-2;12n-1\}$.

In case $g>6n-2$ then all Weierstrass semigroups of genus $g$ containing $<4;4n+1>$ occur at most once.
In case $g=6n-3$ then we can use Lemma \ref{lemma5} to conclude there is a smooth curve $C$ of genus $g$ having two different points $P$ and $Q$ with $\WS (P)=\WS (Q)$ both equal to $<'4;4n+1>\cup \{ 8n-2;12n-1; 12n-5 \}$.

In case $g \equiv 1 \mod 3$ there is exactly one Weierstrass semigroup with first non-gap equal to 3 that occurs at most once.
For other values of $g$ such Weierstrass semigroup does not exist.
In case $g$ is a large integer then $g$ can be written as $6n-m$ with $0\leq m \leq 2n-1$ in many ways.
This implies for all integers $N$ there is a bound $g(N)$ such that for $g \geq g(N)$ there are at least $N$ Weierstrass semigroups of genus $g$ with first non-gap equal to 4 that occur at most once.

\end{example}

\begin{example} \normalfont \label{example2}
Assume $C$ is a smooth curve of genus $g$ and $P$ is a Weierstrass point on $C$ with first non-gap equal to 5.
Note that all Weierstrass semigroups with first non-gap equal to 5 occur as Weierstrass semigroup of some point on some smooth curve (see \cite{ref18}).
Let $b=5n+r$  with $1 \leq r \leq 3$ and $n \geq 2$ in case $r=1$ be another non-gap of $P$ not divisible by 5.
In Lemma \ref{lemmaX} we only have to consider the possibility $e=1$.
In that case all integers at least 20 need to be non-gaps of $P$.
This implies $g \leq 16$.
In case $n=3$ it implies $g \leq 15$ and in case $n=2$ it implies $g \leq 14$.
In case $b=8$ it implies $g \leq 12$ and in case $b=7$ it implies $g \leq 11$.
In all other cases $\vert 5P \vert$ is the only $g^1_5$ on $C$.
In case there are two non-gaps $x$ of $P$ satisfying $5n<x<5(n+1)$ then $g\leq 6n+3$.
In case $5(n-1)+r$ is also a non-gap of $P$ then $g \leq 10n-12+2r$.
This implies $\dim \vert (n+1)5P \vert =n+2$ in case $g>10n-12+2r$ and also in case $g>6n+3$ provided $5(n-1)+r$ is a gap of $P$.

In case $r=1$ and $\dim \vert (n+1)5P \vert = n+2$ a Weierstrass semigroups of genus $g$ containing $<5;5n+1>$ and not containing the set $\{ 10n-3; 15n-7; 15n-2; 20n-11; 20n-6; 20n-1 \}$ occurs at most once unless $g \leq 14$ in case $n=2$.
Consider the case $g=10n-6$.
Then $5n-4$ is a gap of $P$.
In case $n \geq 3$ we also have $6n+3<10n-6$, hence $\dim \vert 5(n+1)P \vert = n+2$.
We can apply Lemma \ref{lemma5} to obtain that all Weierstrass semigroups of genus $10n-6$  containing $<5;5n+1>$ occur at most once except for $<5;5n+1>\cup \{ 10n-3; 15n-7; 15n-2; 20n-11; 20n-6; 20n-1 \}$.
There do exist smooth curves of genus $10n-6$ having two different Weierstrass points having that particular Weierstrass semigroup.

More concretely, for $n \geq 4$ and all non-negative integers $m$; $m'$ satisfying $m \leq n +m'$; $2m \geq m'$; $m+m' >3n+3$ the Weierstrass semigroups $<5;5n+1>\cup \{ 5i+3 : m' \leq i \leq 3n\} \cup \{ 5i+4 : m \leq i \leq 4n \}$ occur at most once.

In case $r=2$ and $n=1$ then as soon as $g<12$ it is possible that $C$ has some $g^1_5$ different from $\vert 5P \vert$.
This is clear, if the plane model $\Gamma$ used in Section 3 has one more singular point $S \neq P$ then the pencil of lines through $S$ induces a base point free $g^1_k$ for some $k \leq 5$ on $C$.
However using different methods it is proved in \cite{ref4} that a curve $C$ of genus $g=11$ has at most one Weierstrass point with Weierstrass semigroup containing $<5;7>$ (see also Lemma \ref{lemmaSection4}).

In case $n \geq 2$ and $\dim \vert (n+1)5P \vert =n+2$ a Weierstrass semigroups of genus g containing $<5;5n+2>$ and not containing the set $\{ 10n-1; 15n+1; 20n-2; 20n+3 \}$ occurs at most once unless $g \leq 14$ in case $n=2$.
Consider the case $g=10n+1$.
Thens $5n-3$ is a gap of $P$, otherwise $<5;5n-3>\subset \WS (P)$ and therefore $g\leq 10n-8$.
In case $n \geq 2$ we also have $6n+3<10n+1$, hence $\dim \vert 5(n+1)P \vert =n+2$.
Again we can obtain Lemma \ref{lemma5} to obtain sharpness of the uniqueness results.
We leave it to the reader to obtain a more concrete description of the Weierstrass semigroups occuring at most once obtained in the paper.

In case $r=3$ and $\dim \vert (n+1)5P \vert = n+2$ we have uniqueness of Weierstrass semigroups of genus $g$ containing $<5;5n+3>$ unless the semigroup contains $\{15n+4; 20n+7\}$ unless $g \leq 12$ in case $n=1$.
Consider the case $g=10n+4$.
Then $5n-2$ is a gap of $P$ and $6n+3 \leq 10n+4$, hence $\dim \vert 5(n+1)P \vert =n+2$ for $n \geq 1$.
Again we can apply Lemma \ref{lemma5} obtaining sharpness of the uniqueness results.
A more concrete description of the Weierstrass semigroups occuring at most once is very similar to the description obtained in Example \ref{example1}

\end{example}

\begin{example} \normalfont \label{example3}
For the case $a=6$ we only need to consider a non-gap $6n+1$ with $n \geq 2$.
In case there is a non-gap between 6n+1 and 6n+6 then $g \leq 9n+1$.
The bound in Corollary \ref{corollary2} is $g>15n-10$.
We consider the case $g=15n-10$.

Since $9n+1<15n-10$ there are no non-gaps between $6n+1$ and $6n+6$ in case $g=15n-10$.
Also in case $6n-5$ would be a non-gap then $g \leq 15n-15$.
This implies $\dim \vert (n+1)6P \vert = n+2$.
Assume $C$ has another $g^1_6$ different from $\vert 6P \vert$.
In Lemma \ref{lemmaX} we need to consider the cases $e=1$; 2 and 3.
In case $e=1$ then all integers at least 30 are non-gaps.
This implies the existence of more than 10 non-gaps outside $<6;6n+1>$ in case $n \geq 3$.
In case $e=2$ then all even integers at least equal to 12 are non-gaps.
This implies the existence of $6n-4$ non-gaps outside $<6;6n+1>$.
In case $n \geq 3$ we obtain more than 10 non-gaps outside $<6;6n+1>$.
In case $e=3$ then all integers divisible by 3 and at least equal to 6 are non-gaps.
In case $n\geq 3$ this implies 3n-1 non-gaps outside $<6;6n+1>$.
In case $n\geq 4$ it implies the existence of more than 10 non-gaps outside $<6;6n+1>$.
Therefore only in case $n \geq 4$ the linear system $g^1_6$ is unique.

This example gives an illustration of the fact that $e>1$ in Lemma \ref{lemmaX} can impose conditions in applying the results of this paper.
We can only make the following conclusion can only be made in case $n\geq 4$.
For $g>15n-10$ all Weierstrass semigroups of genus $g$ containing $<6;6n+1>$ occur at most once.
In case $g=15n-10$ the only Weierstrass semigroup of genus $g$ containing $<6;6n+1>$ and occuring more than once is equal to $<6;6n+1>\cup \{12n-4; 18n-9; 18n-3; 24n-12; 24n-8; 24n-4; 30n-19; 30n-13; 30n-7; 30n-1 \}$.
\end{example}

In case $n=1$ in general Lemma \ref{lemmaX} does not imply uniqueness of $g^1_a$ in case $g\geq \frac{(a-1)(b-a+r)}{2}$.
In this case we can use another argument to conclude uniqueness of Weierstrass semigroups containing $<a;a+r>$ with $1 \leq r \leq a-2$ and $(a,r)=1$.

\begin{lemma} \normalfont \label{lemmaSection4}
Fix an integer $r\geq 2$.
There is a bound $A(r)$ such that in case $a\geq A(r)$, $(a,r)=1$ and $g>\frac{(a+r-1)(a-r)}{2}$ then a smooth curve $C$ of genus $g$ has at most one Weierstrass point whose Weierstrass semigroup contains $<a;a+r>$.
\end{lemma}

\begin{proof}
As a matter of fact, there is a genus bound $g(a;r)$ obtained in \cite{ref19} , Theorem 4.3, such that in case $C$ is a smooth curve of genus $g \geq g(a;r)$ then $C$ has at most one simple $g^2_{a+r}$.
This genus bound behaves like a polynomial in $a$ with highest order term $\frac{a^2}{3}$.
Since $\frac{(a-1)(b-1)}{2}$ is polynomial with highest order term $\frac{a^2}{2}$ it implies that for small $e \geq 1$ one has $\frac{(a-1)(b-1)}{2}-e\geq g(a;r)$ if $a>>0$.
In such case, if $P$ and $Q$ are two different Weierstrass points on $C$ such that $\WS (P)$ and $\WS (Q)$ both contain $<a;a+r>$ then we obtain $\vert (a+r)P \vert = \vert (a+r)Q \vert$.
In particular $P$ and $Q$ have to induce both a cusp of type $(r;a+r)$ on the same plane model $\Gamma$ of $C$ (as considered in Section 3).
This implies
\[
g\leq \frac{(a+r-1)(a-r)}{2} \text { .}
\]
This genus bound is polynomial with highest order term equal to $\frac{a^2}{2}$, hence for $a>>0$ this bound is larger than $g(a;r)$.
This implies that in case $g>\frac{(a+r-1)(a-r)}{2}$ and $a>>0$ then a curve $C$ of genus $g$ has at most one Weierstrass point with Weierstrass semigroup containing $<a;a+r>$.
\end{proof}

In case $r=2$ this is part of the arguments used in \cite{ref4} and using more detailled arguments one obtains clear and good genus bounds.
It should be noted that in case $a>>0$ this bound on $g$ is sharp.
One can make use of plane rational curves having two cusps of type $(r;a+r)$ as follows.
Choose $P$ and $Q$ different points on $\mathbb{P}^1$.
Let $E$ be a general effective divisor of degree $a-r$ and consider the pencil $rP+<aP;rQ+E>$, a $g^1_{a+r}$ on $\mathbb{P}^1$.
Then take $g^2_b=<(a+r)Q;rP+<aP;rQ+E>>$.
Then using arguments as those used in \cite{ref1}, Section 3, one can show that there exists a plane curve of degree $a+r$ having exactly two cusps of type $(r;a+r)$ and no other singularities.
The normalisation $C$ of this curve has genus $\frac{(a+r-1)(a-1)}{2}$ and has two different Weierstrass points whose Weierstrass semigroups contain $<a;a+r>$.

\end{document}